\documentclass[a4paper,11pt,french,english]{amsart}
\usepackage[english]{babel}
\usepackage[utf8]{inputenc}
\usepackage{lmodern}
\usepackage[T1]{fontenc}
\usepackage{amsmath,amsthm,amssymb,amsfonts}
\usepackage{amscd}
\usepackage{pstricks}
\usepackage{pst-node}
\usepackage[a4paper,left=3.4cm,right=3.4cm,top=4cm,bottom=4cm]{geometry}
\usepackage{enumerate}
\usepackage{enumitem}
\usepackage[linktocpage=true]{hyperref}
\usepackage{url}
\usepackage{csquotes}

\newcommand{\Sy}{\mathrm{Sym}(\Omega)}
\newcommand{\Syd}{\mathrm{Sym}(d)}

\newcommand{\aut}{\mathrm{Aut}(T_d)}

\title{Simple groups and irreducible lattices in wreath products}

\author{Adrien Le Boudec}
\thanks{This work was carried out when the author was F.R.S.-FNRS Postdoctoral Researcher. Current affiliation: CNRS, UMPA - ENS Lyon. Partially supported by ANR-14-CE25-0004 GAMME}
\address{UCLouvain, IRMP,	Chemin du Cyclotron 2, 1348 Louvain-la-Neuve, Belgium}
\address{CNRS, Unité de Mathématiques Pures et Appliquées,	ENS-Lyon, France}
\email{adrien.le-boudec@ens-lyon.fr}
\date{January 22, 2020}

\theoremstyle{plain}
\newtheorem{thm}{Theorem}[section]
\newtheorem{prop}[thm]{Proposition}
\newtheorem{cor}[thm]{Corollary}
\newtheorem{lem}[thm]{Lemma}

\newtheorem*{thm-intro}{Theorem}

\theoremstyle{definition}
\newtheorem{defi}[thm]{Definition}

\newtheorem{rmq}[thm]{Remark}

\begin{document}

\maketitle

\begin{abstract}
We consider the finitely generated groups acting on a regular tree with almost prescribed local action. We show that these groups embed as cocompact irreducible lattices in some locally compact wreath products. This provides examples of finitely generated simple groups quasi-isometric to a wreath product $C \wr F$, where $C$ is a finite group and $F$ a non-abelian free group.
\end{abstract}



\section{Introduction}

Let $W_d$ be the free product of $d \geq 3$ copies of the cyclic group $C_2$ of order $2$. If $x_1,\ldots,x_d$ are generators of the copies of $C_2$, the Cayley graph of $W_d$ with respect to $X = \left\{x_1,\ldots,x_d\right\}$ is a $d$-regular tree $T_d$. Given a finite group $A$ of cardinal $n \geq 2$, let $A \wr W_d$ be the restricted wreath product of $A$ and $W_d$, and denote by $\mathcal{C}_{n,d}$ the Cayley graph of $A \wr W_d$ with respect to the generating subset $S = A \cup X$. The graph $\mathcal{C}_{n,d}$ is sometimes called the \textbf{lamplighter graph} over $T_d$. Note that the graph $\mathcal{C}_{n,d}$ does not really depend on the group $A$, but only on the cardinality of $A$. The goal of this short article is to show that that certain finitely generated groups $\Gamma \leq \aut$ acting non-properly on $T_d$, also admit a proper and cocompact action on the graph $\mathcal{C}_{n,d}$.

To any permutation group $F \leq \Syd$, there is an associated group of tree automorphisms $G(F)$ with local action prescribed almost everywhere by $F$. The group $G(F)$ consists of all $g \in \aut$ such that the local permutation of $g$ at the vertex $v$ belongs to $F$ for all but finitely many $v$. More generally given $F \lneq F' \leq \Syd$, $G(F,F')$ is the group tree automorphisms with local action prescribed everywhere by $F'$ and  almost everywhere by $F$. See \S \ref{subsec-thegroups} for formal definitions. When the permutation group $F$ is semi-regular, the group $G(F,F')$ is a finitely generated group. Recall that a permutation group is \textbf{semi-regular} if it has trivial point stabilizers. The action of $G(F,F')$ on the tree is not proper, and stabilizers of vertices are infinite locally finite subgroups. This family of groups share some common properties with the irreducible lattices in the product of two trees constructed by Burger--Mozes \cite{BM-IHES-2}: both families contain instances of groups that embed densely in some universal group \cite[\S 3.2]{BM-IHES} and that are virtually simple. Here we show that the groups $G(F,F')$ also embed as irreducible lattices in some non-discrete locally compact groups, but these are wreath products instead of direct products.

If $H$ is a group acting on a set $X$, and $A$ is a subgroup of a group $B$, the \textbf{semi-restricted permutational wreath product} $B \wr_X^A H$, introduced by Cornulier in \cite{Cor-semi-wreath}, is the semi-direct product $B^{X,A} \rtimes H$, where $B^{X,A}$ is the set of functions $f: X \rightarrow B$ such that $f(x) \in A$ for all but finitely many $x \in X$, and $H$ acts on $B^{X,A}$ in the usual way: $(h \cdot f)(x) = f(h^{-1}x)$. This definition somehow interpolates between the restricted and the unrestricted permutational wreath products, which correspond respectively to $A=1$ (in which case we will write $B \wr_X H$) and $A=B$. When $B,H$ are locally compact and $A$ is compact open in $B$, there is a natural locally compact group topology on $B \wr_X^A H$ (see \S \ref{subsub-wr}).

We call a lattice $\Gamma$ in $B \wr_X^A H$ \textbf{irreducible} if $\Gamma$ has a non-discrete projection to $H$. The terminology is motivated by the fact that this definition prevents $\Gamma$, and more generally any subgroup commensurable with $\Gamma$, from being of the form $\Gamma_1 \rtimes \Gamma_2$, where $\Gamma_1$ and $\Gamma_2$ are lattices in $B^{X,A}$ and $H$.

\begin{defi}
For $n \geq 2$, we denote by $G_{n,d}$ the semi-restricted wreath product $G_{n,d} = \mathfrak{S}_{n} \wr_{T_d}^{\mathfrak{S}_{n-1}} \aut$, where $\mathfrak{S}_{k}$  is the symmetric group on $k$ elements.
\end{defi}

The natural action of the unrestricted wreath product $ \mathfrak{S}_{n} \wr_{T_d}^{\mathfrak{S}_{n}} \aut$ on the set of functions $T_d \to \left\{1,\ldots,n\right\}$ induces an action of the group $G_{n,d}$ by graph automorphisms on $\mathcal{C}_{n,d}$, and this action is proper and cocompact (see Proposition \ref{prop-act-Gnd-X}).


Our main result is the following:

\begin{thm} \label{thm-intro-coc-and-graph}
Let $d \geq 3$, $F \lneq F' \leq \Syd$ permutation groups such that $F$ is semi-regular, and $n$ the index of $F$ in $F'$. Then the group $G(F,F')$ embeds as an irreducible cocompact lattice in the semi-restricted permutational wreath product $G_{n,d}$. So the group $G(F,F')$ acts properly and cocompactly on the graph $\mathcal{C}_{n,d}$.
\end{thm}

It immediately follows from the theorem that the group $G(F,F')$ is quasi-isometric to $C_n \wr W_d$, where $C_n$ is the cyclic group of order $n$, since they both act properly and cocompactly on the same graph.

The embedding of $G(F,F')$ in $G_{n,d} = \mathfrak{S}_{n}^{V_d,\mathfrak{S}_{n-1}} \rtimes \aut$ is not the inclusion in the subgroup $1 \rtimes \aut$, but a twisted embedding associated to the cocycle given by the local action on $T_d$. See Section \ref{sec-proof} for details.

The case $n=2$ is particular as the group $G_{2,d}$ is actually a restricted wreath product, and in this situation the group $G(F,F')$ is an irreducible cocompact lattice in $C_2 \wr_{V_d} \aut = \left( \oplus_{V_d} C_2 \right) \rtimes \aut$. 

\begin{rmq}
The groups $G(F,F')$ are instances of countable discrete groups $\Gamma$ with a continuous Furstenberg uniformly recurrent subgroup (URS). Lattice embeddings for this class of groups (i.e.\ given such a group $\Gamma$, study the locally compact groups that can contain a copy of $\Gamma$ as a lattice) are studied in \cite{LB-am-urs-lat}; and we refer the reader to \cite{LB-am-urs-lat} for the above terminology. So on the one hand Theorem \ref{thm-intro-coc-and-graph} motivates the problems considered in \cite{LB-am-urs-lat}, and on the other hand all the results obtained in \cite{LB-am-urs-lat} apply to the family of groups $G(F,F')$ (see notably Corollary 1.3 in \cite{LB-am-urs-lat}). 
\end{rmq}

\subsection*{Applications}

Recall that the property of being virtually simple is not invariant by quasi-isometry. Indeed the lattices constructed by Burger and Mozes in \cite{BM-IHES-2} show that a virtually simple finitely generated group may have the same Cayley graph as a product of two finitely generated free groups. Theorem \ref{thm-intro-coc-and-graph} together with simplicity results from \cite[\S 4.2]{LB-ae} (see also \cite[Prop.\ 6.8]{LB-am-urs-lat}) provide another illustration of this fact, namely finitely generated simple groups having the same Cayley graph as a wreath product. The wreath product construction is already known to be a source of examples of finitely generated groups whose algebraic properties are not reflected in their Cayley graphs. Two wreath products $B_1 \wr \Gamma$ and $B_2 \wr \Gamma$ may have isometric or bi-Lipschitz Cayley-graphs, one being solvable or torsion free, while no finite index subgroup of the second has these properties \cite{Ersch-solv-noQI}. The phenomenon that we exhibit here is nonetheless very different, in the sense that it provides finitely generated groups with isometric Cayley graphs such that one is a wreath product, but the other is simple (and hence not commensurable with a wreath product). See Theorem \ref{thm-coc-and-graph}.

Recall that for finitely generated groups, being amenable is an invariant of quasi-isometry, and an invariant of measure equivalence. By contrast, Theorem \ref{thm-intro-coc-and-graph} implies:

\begin{cor} \label{cor-intro}
Among finitely generated groups, the property of having infinite amenable radical is invariant neither by quasi-isometry nor by measure equivalence.
\end{cor}

These examples simultaneously show that having an infinite elliptic radical is also not invariant by quasi-isometry. Recall that the elliptic radical of a discrete group is the largest locally finite normal subgroup.

Recall that by a theorem of Eskin--Fisher--Whyte, any finitely generated group $\Gamma$ that is quasi-isometric to a wreath product $C \wr \mathbb{Z}$, where $C$ is a finite group, must act properly and cocompactly on a Diestel-Leader graph $\mathrm{DL}(m,m)$ \cite{EFW-rig-sol,EFW-coarseI,EFW-coarseII}. By the algebraic description of the isometry groups of these graphs given in \cite{hor-prod-tree} (see also \cite{cross-lamp}), this implies in particular that $\Gamma$ has a subgroup of index at most two that is (locally finite)-by-$\mathbb{Z}$. By contrast, Theorem \ref{thm-intro-coc-and-graph} shows that this rigidity fails in the case of $C \wr \mathbb{F}_{k}$ when $k \geq 2$.

The proof of Theorem \ref{thm-intro-coc-and-graph} is given in Section \ref{sec-proof}, which is the core of the article. In Section \ref{sec-comments} we make additional observations regarding other consequences of the phenomena exhibited in Theorem \ref{thm-intro-coc-and-graph} (see notably Propositions \ref{prop-K(F,F')-lattice} and \ref{prop-strong-am-coco}).

\subsection*{Acknowledgements}
I am indebted to Alain Valette for a decisive remark made in Neuch\^{a}tel in May 2015, which eventually led to Theorem \ref{thm-intro-coc-and-graph}. I am also grateful to a referee for pointing out that  Corollary \ref{cor-intro} was also new in the setting of measure equivalence.


\section{The proof}  \label{sec-proof}

\subsection{Groups acting on trees with almost prescribed local action} \label{subsec-thegroups}

Let $\Omega$ be a set of cardinality $d \geq 3$, and let $T_d$ be a $d$-regular tree with vertex set $V_d$ and edge set $E_d$. We fix a coloring $c: E_d \rightarrow \Omega$ such that around every vertex $v$ the map $c$ induces a bijection between edges in the star around $v$ and $\Omega$. For $g \in \aut$ and $v \in V_d$, the action of $g$ on the star around $v$ gives rise to a permutation $\sigma(g,v)$ of $\Omega$, called the local permutation of $g$ at $v$. These local permutations verify the rule \begin{equation} \label{eq:loc-perm} \sigma(gh,v) = \sigma(g,hv) \sigma(h,v) \end{equation} for every $g,h \in \aut$ and $v \in V_d$.

Given a permutation group $F \leq \Sy$, the group $U(F)$ is the group of automorphisms of $T_d$ whose local action is prescribed by $F$, i.e.\ the $g \in \aut$ such that $\sigma(g,v) \in F$ for all $v$ \cite{BM-IHES}. It is a closed cocompact subgroup of $\aut$.

\begin{defi}
Given $F \leq F'\leq \Sy$, we denote by $G(F,F')$ the group of automorphisms $g \in \aut$ such that the local action of $g$ is prescribed to be in $F'$ everywhere, and in $F$ almost everywhere: $\sigma(g,v) \in F'$ for all $v$ and $\sigma(g,v) \in F$ for all but finitely many $v$.
\end{defi}

We also denote by $G(F,F')^\ast$ the subgroup of index two in $G(F,F')$ preserving the bipartition of $T_d$.

Note that requiring that the local permutations belong to $F'$ does not a priori ensure that any element of $F'$ appears as a local permutation of an element of $G(F,F')$. However this is true under the assumption that $F'$ preserves the orbits of $F$ \cite[Lem.\ 3.3]{LB-ae}. In the sequel we will systematically assume that this assumption is fulfilled.

Contrary to the group $U(F)$, the group $G(F,F')$ is no longer closed in $\aut$ when $F \neq F'$. Actually its closure is equal to the group $U(F')$ \cite[Prop.\ 3.5]{LB-ae}. The group $G(F,F')$ is countable if and only if the permutation group $F$ is semi-regular. When this is so, $G(F,F')$ is actually a finitely generated group \cite[Cor.\ 3.8]{LB-ae}. Its action on the tree is not proper, and stabilizers of vertices are infinite locally finite subgroups \cite[\S 3.1]{LB-ae}.

\subsection{Locally compact wreath products} \label{subsub-wr}

Let $X$ be a set, $B$ a group and $A$ a subgroup of $B$. We will denote by $B^{X,A}$ the set of functions $f: X \rightarrow B$ such that $f(x) \in A$ for all but finitely many $x \in X$. Note that $B^{X,A}$ is a group.

\begin{defi}[\cite{Cor-semi-wreath}]
If $H$ is a group acting on $X$, the \textbf{semi-restricted permutational wreath product} $B \wr_X^A H$ is the semi-direct product $B^{X,A} \rtimes H$, where $h \in H$ acts on $f \in B^{X,A}$ by $(hf)(x) = f(h^{-1}x)$.
\end{defi}

The extreme situations when $A=1$ and when $A=B$ correspond respectively to the restricted and the unrestricted wreath product. When $A=1$, we shall write $B \wr_X H$ for the restricted wreath product. Also for simplicity we will sometimes say \enquote{wreath product} $B \wr_X^A H$ instead of \enquote{semi-restricted permutational wreath product}.

When $A$ is a compact group and $H$ a locally compact group acting continuously on $X$, the group $A^{X} \rtimes H$ is a locally compact group for the product topology. If moreover $B$ is locally compact and $A$ is compact open in $B$, there is a natural locally compact group topology on $B \wr_X^A H$, defined by requiring that the inclusion of $A^{X} \rtimes H$ in $B \wr_X^A H$ is continuous and open. See \cite[Sec.\ 2]{Cor-semi-wreath}. 

\medskip

We are interested in the study of certain lattices in some locally compact groups $B \wr_X^A H$. A few remarks are in order:

\begin{lem} \label{lem-red-lat}
Let $A,B,X$ and $H$ as above.
\begin{enumerate}[label=(\alph*)]
\item Assume $\Gamma_1$ is a lattice in $B^{X,A}$ and $\Gamma_2$ is a lattice in $H$ that normalizes $\Gamma_1$. Then $\Gamma_1 \rtimes \Gamma_2$ is a lattice in $B \wr_X^A H$.
\item \label{item-exist-lat-wr} For $B \wr_X^A H$ to contain a lattice, it is necessary that $H$ contains a lattice.
\end{enumerate}
\end{lem}

\begin{proof}
For the first statement, see \cite[Lem.\ I.1.6-7]{Raghu-dis-sb}. For the second statement, observe that if $\Gamma \leq B \wr_X^A H$ is a lattice, the intersection $\Gamma_A = \Gamma \cap (A^{X} \rtimes H)$ is a lattice in $A^{X} \rtimes H$ since $A^{X} \rtimes H$ is open in $B \wr_X^A H$. The subgroup $A^{X}$ being compact, the projection of $\Gamma_A$ to $H$ is discrete, and hence is a lattice in $H$.
\end{proof}

Recall that there are various notions of irreducibility for a lattice in a direct product of groups. In general whether all these notions coincide depends on the context. We refer to \cite[2.B]{CaMo-KM} and \cite[4.A]{CaMo-Iso-disc} for detailed discussions. In the setting of wreath products, we will use the following terminology: 

\begin{defi}
A lattice $\Gamma$ in $B \wr_X^A H$ is an \textbf{irreducible lattice} if $\Gamma$ has a non-discrete projection to the group $H$.
\end{defi}

This definition implies that neither $\Gamma$ nor its finite index subgroups can be of the form $\Gamma_1 \rtimes \Gamma_2$ as in Lemma \ref{lem-red-lat}.

\begin{lem} \label{lem-all-lat-irr}
If the group $B^{X,A}$ does not contain any lattice, then any lattice in $B \wr_X^A H$ is irreducible.
\end{lem}

\begin{proof}
If $\Gamma \leq B \wr_X^A H$ is a lattice with a discrete projection $\Lambda$ to $H$, then the subgroup $B \wr_X^A \Lambda$ contains $\Gamma$ as a lattice, and it follows that $\Gamma$ intersects the subgroup $B^{X,A}$, which is open in $B \wr_X^A \Lambda$, along a lattice of $B^{X,A}$.
\end{proof}

\begin{rmq} \label{rmq-wreath-no-lattice}
Of course if $B$ admits no lattice then the same holds for $B^{X,A}$. More interestingly, there are finite groups $A,B$ for which $B^{X,A}$ fails to admit any lattice (provided $X$ is infinite). This is for instance the case when $A \neq B$ and any non-trivial element of $B$ has a non-trivial power in $A$. Consequently all lattices in $B \wr_X^A H$ are irreducible by Lemma \ref{lem-all-lat-irr} (for arbitrary $H$).
\end{rmq}

\begin{proof}[Proof of Remark \ref{rmq-wreath-no-lattice}]
We claim that the above condition on $A,B$ actually implies that $B^{X,A}$ has no infinite discrete subgroup. For every finite $\Sigma \subset X$, we write $O_\Sigma \leq B^{X,A}$ for the subgroup vanishing on $\Sigma$, and $U_\Sigma = O_\Sigma \cap A^X$. Assume $\Gamma$ is a discrete subgroup of $B^{X,A}$, so that there is a finite $\Sigma \subset X$ such that $\Gamma \cap U_\Sigma = 1$. The assumption on $A,B$ is easily seen to imply that any non-trivial subgroup of $O_\Sigma$ intersects $U_\Sigma$ non-trivially. Therefore $\Gamma \cap O_\Sigma = 1$, and $O_\Sigma$ being of finite index in $B^{X,A}$, $\Gamma$ is finite.
\end{proof}

It should be noted that the existence of an irreducible lattice in $B \wr_X^A H$ forces $H$ to be non-discrete and $B$ to be non-trivial. However this does not force $B^{X,A}$ to be non-discrete, and as we will see below, interesting examples already arise when $B$ is finite and $A$ is trivial.

\subsection{The proof of Theorem \ref{thm-intro-coc-and-graph}} \label{subsubsec-irr-wr}

Let $n \geq 2$ and $d \geq 3$. We denote by $\Sigma_n$ the set of integers $\left\{0,\ldots,n-1\right\}$, and by $\mathfrak{S}_{n}$ the group of permutations of $\Sigma_n$. For $\sigma \in \mathfrak{S}_{n}$ and $i \in \Sigma_n$, we will write $\sigma \cdot i$ the action of $\sigma$ on $i$. The stabilizer of $0 \in \Sigma_n$ in $\mathfrak{S}_{n}$ is obviously isomorphic to $\mathfrak{S}_{n-1}$, and by abuse of notation we will denote it $\mathfrak{S}_{n-1}$. In particular when viewing $\mathfrak{S}_{n-1}$ as a subgroup of $\mathfrak{S}_{n}$, we will always implicitly mean that $\mathfrak{S}_{n-1}$ is the subgroup of $\mathfrak{S}_{n}$ acting only on $\left\{1,\ldots,n-1\right\}$.

\begin{defi}
We will denote by $G_{n,d}$ the wreath product $\mathfrak{S}_{n} \wr_{V_d}^{\mathfrak{S}_{n-1}} \aut$. 
\end{defi}

Groups of the form $G_{n,d}$ were considered in \cite[Ex.\ 2.6]{Cor-semi-wreath}. We will denote $U_{n,d} = \mathfrak{S}_{n}^{V_d, \mathfrak{S}_{n-1}}$, so that $G_{n,d} = U_{n,d} \rtimes \aut$. We endow $G_{n,d}$ with the topology such that sets of the form $((\sigma_v)U_1,\gamma U_2)$ form a basis of neighbourhoods of $((\sigma_v),\gamma)$, where $U_1$ and $U_2$ belong to a basis of the identity respectively in $\mathfrak{S}_{n-1}^{V_d}$ and in $\aut$. This defines a totally disconnected locally compact group topology on $G_{n,d}$ (see \cite[Prop.\ 2.3]{Cor-semi-wreath}). We note that the case $n=2$ is somehow particular, as $U_{2,d}$ is a discrete subgroup of $G_{2,d}$, and $G_{2,d}$ is just the restricted wreath product $G_{2,d} = C_2 \wr_{V_d} \aut = \left( \oplus_{V_d} C_2 \right) \rtimes \aut$.

Our goal in this paragraph is to prove the following, which is slightly more precise than Theorem \ref{thm-intro-coc-and-graph}:

\begin{thm} \label{thm-coc-and-graph}
Let $d \geq 3$, $F \lneq F' \leq \Syd$ permutation groups such that $F$ is semi-regular, and $n$ the index of $F$ in $F'$. Then:
\begin{enumerate}[label=(\alph*)]
\item The group $G(F,F')$ embeds as an irreducible cocompact lattice in the semi-restricted permutational wreath product $G_{n,d}$.
\item When $F$ is regular, the finitely generated group \[ \Gamma_{n,d} = C_n \wr_{V_d} (C_d \ast C_d) = (C_n^2 \, \wr \, \mathbb{F}_{d-1}) \rtimes C_d \] and $G(F,F')^\ast$ have isometric Cayley graphs.
\end{enumerate}
\end{thm}

Recall that a permutation group is \textbf{regular} if it is transitive and semi-regular.

We denote by $\Sigma_n^{(V_d)}$ the set of functions $f: V_d \rightarrow \Sigma_n$ with finite support, where the support of $f$ is the set of $v$ such that $f(v) \neq 0$. We will also write $f_v$ for the image of $v$ by $f$, and sometimes use the notation $(f_v)$ for the function $f$.

We consider the graph $X_{n,d}$ whose set of vertices is the set of pairs $(f,e)$, where $f$ belongs to $\Sigma_n^{(V_d)}$ and $e \in E_d$, and edges emanating from a vertex $(f,e)$ are of two types:
\begin{itemize}
	\item type 1: $(f,e')$ is connected to $(f,e)$ if $e' \in E_d$ is a neighbour of $e$ (i.e.\ if $e$ and $e'$ share exactly one vertex);
	\item type 2: $(f',e)$ is connected to $(f,e)$ if the function $f'$ is obtained from $f$ by changing the value at exactly one vertex of $e$.
\end{itemize}

Note that since any $e \in E_d$ has $2(d-1)$ neighbours and $\Sigma_n$ has cardinality $n$, every vertex of $X_{n,d}$ has $2(d-1)$ neighbours of type 1 and $2(n-1)$ neighbours of type 2. The graph $X_{n,d}$ is tightly related to the wreath product of the complete graph on $n$ vertices with the tree $T_d$, see \S \ref{subsec-variations}.

\begin{prop} \label{prop-act-Gnd-X}
The group $G_{n,d}$ acts by automorphisms on the graph $X_{n,d}$ by preserving the types of edges. Moreover the action is faithful, continuous, proper and transitive on the set of vertices. 
\end{prop}

\begin{proof}
The group $G_{n,d}$ is a subgroup of the unrestricted permutational wreath product of $\mathfrak{S}_{n}$ and $\aut$. The latter group has a faithful action on the set of functions $V_d \rightarrow \Sigma_n$, given by \[ ((\sigma_v),\gamma) \cdot (f_v) = (\sigma_v \cdot f_{\gamma^{-1}v}). \] The group $G_{n,d}$ preserves $\Sigma_n^{(V_d)}$ because $\sigma_v$ fixes $0$ almost surely if $(\sigma_v)$ belongs to $U_{n,d}$. Now the projection from $G_{n,d}$ onto $\aut$ induces an action of $G_{n,d}$ on the set $E_d$, and we will consider the diagonal action of $G_{n,d}$ on $\Sigma_n^{(V_d)} \times E_d$. In other words, if $g = ((\sigma_v),\gamma) \in G_{n,d}$, and $((f_v), e)\in X_{n,d}$,  \begin{equation} \label{eq:action} g \cdot \left( (f_v), e \right) = \left((\sigma_v \cdot f_{\gamma^{-1}v}),\gamma e \right). \end{equation}

Fix $x = ((f_v), e)\in X_{n,d}$ and $g = ((\sigma_v),\gamma) \in G_{n,d}$, and let $x'$ be a neighbour of $x$. If $x'$ is of type 1, then we have $x' = ((f_v), e')$, where $e$ and $e'$ share a vertex $w$ in $T_d$. Then $\gamma e$ and $\gamma e'$ have the vertex $\gamma w$ in common, so that by the formula (\ref{eq:action}), $g\cdot x'$ is a neighbour of type 1 of $g\cdot x$ in $X_{n,d}$. Now if $x'$ is of type 2, then we may write $x' = ((f_v'), e)$ with $f_v' = f_v$ if and only if $v \neq w$, where $w$ is one of the two vertices of $e$. It follows that $\sigma_v \cdot f_{\gamma^{-1}v} = \sigma_v \cdot f_{\gamma^{-1}v}'$ if and only if $v \neq \gamma w$, so that by (\ref{eq:action}) $g\cdot x'$ is a neighbour of type 2 of $g\cdot x$. This shows that the action is by graph automorphisms and preserves the types of edges.

\begin{lem} \label{lem-stab-vert-X}
Let $e \in E_d$, and let $K$ be the stabilizer of $e$ in $\aut$. Then the stabilizer of the vertex $((0),e)$ in $G_{n,d}$ is the compact open subgroup $\prod \mathfrak{S}_{n-1} \rtimes K$.
\end{lem}

\begin{proof}
That $g = ((\sigma_v),\gamma)$ fixes $((0),e)$ exactly means by (\ref{eq:action}) that $\sigma_v$ fixes $0$ for all $v$ and that $\gamma$ fixes $e$.
\end{proof}

So the fact that the action is continuous and proper follows from the lemma, and the transitivity on the set of vertices is an easy verification.
\end{proof}

Consider now the free product $C_d \ast C_d$ of two cyclic groups of order $d$, acting on its Bass-Serre tree $T_d$ with one orbit of edges and two orbits of vertices. Denote by $C_n$ the cyclic subgroup of $\mathfrak{S}_{n}$ generated by the cycle $(0,\ldots,n-1)$, and set \[ \Gamma_{n,d}:= C_n \, \wr_{V_d} \, (C_d \ast C_d) \leq G_{n,d}. \] Remark that $C_d \ast C_d$ has a split morphism onto $C_d$, whose kernel acts on $T_d$ with two orbits of vertices and is free of rank $d-1$. Therefore $\Gamma_{n,d}$ splits as $\Gamma_{n,d} = (C_n^2 \, \wr \, \mathbb{F}_{d-1}) \rtimes C_d$.

\begin{lem} \label{lem-stand-act}
$\Gamma_{n,d} \leq G_{n,d}$ acts freely transitively on the vertices of $X_{n,d}$.
\end{lem}

\begin{proof}
This is clear: the image of the vertex $((0),e)$ by an element $((\sigma_v),\gamma)$ is  $((\sigma_v \cdot 0), \gamma e)$, so both transitivity and freeness follow from the fact that the actions of $C_n$ on $\Sigma_n$ and of $C_d \ast C_d$ on $E_d$ have these properties.
\end{proof}

We now explain how the groups $G(F,F')$ act on the graphs $X_{n,d}$. In the sequel $F,F'$ denote two permutation groups on a set $\Omega$ of cardinality $d$ such that $F \leq F'$ and $F'$ preserves the orbits of $F$, and we denote by $n$ the index of $F$ in $F'$.

Fix a bijection between $\Sigma_n = \left\{0,\ldots,n-1\right\}$ and $F' / F$, such that $0$ is sent to the class $F$. The action of $F'$ on the coset space $F' / F$ induces a group homomorphism \[ \alpha: F' \rightarrow \mathfrak{S}_{n}, \] such that $\alpha(F)$ lies inside $\mathfrak{S}_{n-1}$. For $\gamma \in G(F,F')$ and $v \in V_d$, write \[ \rho_{\gamma,v} = \alpha (\sigma(\gamma,\gamma^{-1}v)) \in \mathfrak{S}_{n}. \] Note that $\rho_{\gamma,v} \in \mathfrak{S}_{n-1}$ if and only if $\sigma(\gamma,\gamma^{-1}v) \in F$. We also denote by $\rho_{\gamma} = (\rho_{\gamma,v})$.

Recall that in general the group $G(F,F')$ admits a locally compact group topology, that is defined by requiring that the inclusion of $U(F)$ is continuous and open \cite[\S 3.1]{LB-ae}. This topology is the discrete topology if and only if the permutation group $F$ is semi-regular.


\begin{prop} \label{prop-embed-g(f,f')-wr}
Let $F \leq F' \leq \Sy$, and $n$ the index of $F$ in $F'$. The map $\varphi: G(F,F') \rightarrow G_{n,d}$, $\gamma \mapsto \varphi(\gamma) = \left( \rho_{\gamma}, \gamma \right)$, is a well-defined group morphism that is injective, continuous, and with a closed and cocompact image.
\end{prop}

\begin{proof}
The map $\varphi$ is well-defined because $\rho_{\gamma,v} \in \mathfrak{S}_{n-1}$ for all but finitely many $v$, so that we indeed have $\rho_{\gamma} \in \mathfrak{S}_{n}^{V_d, \mathfrak{S}_{n-1}}$. The fact that $\varphi$ is a group morphism follows from the cocycle identity (\ref{eq:loc-perm}) satisfied by local permutations. Indeed for $\gamma, \gamma' \in G(F,F')$ we have $\varphi(\gamma) \varphi(\gamma') = ( \psi,\gamma \gamma')$ with 
\[ \begin{aligned} \psi_v & = \rho_{\gamma,v} \rho_{\gamma',\gamma^{-1}v} \\ & = \alpha (\sigma(\gamma,\gamma^{-1}v)) \alpha (\sigma(\gamma',\gamma'^{-1}\gamma^{-1}v)) \\ & = \alpha (\sigma(\gamma \gamma',(\gamma \gamma')^{-1}v)) \\ & = \rho_{\gamma \gamma',v}, \end{aligned} \]
so $\psi = \rho_{\gamma \gamma'}$ and $\varphi(\gamma) \varphi(\gamma') = \varphi(\gamma \gamma')$.

Injectivity of $\varphi$ is clear since the composition with the projection to $\aut$ is injective. The preimage in $G(F,F')$ of the open subgroup $\mathfrak{S}_{n-1}^{V_d} \rtimes \aut$ is the subgroup $U(F)$, which is open in $G(F,F')$ by definition of the topology, so it follows that the map $\varphi$ is continuous. Also the intersection between $\mathrm{Im}(\varphi)$ and the open subgroup $\mathfrak{S}_{n-1}^{V_d} \rtimes \aut$ is $\varphi(U(F))$, and it is easy to check that the latter is indeed a closed subgroup of $G_{n,d}$, so it follows that $\mathrm{Im}(\varphi)$ is closed in $G_{n,d}$. The fact that $\mathrm{Im}(\varphi)$ is cocompact will follow from Proposition \ref{prop-act-Gnd-X} and Proposition \ref{prop-g(f,f')-trans-som} below.
\end{proof}

In the sequel for simplicity we will also write $G(F,F')$ for the image of \[ \varphi: G(F,F') \rightarrow G_{n,d}, \, \, \gamma \mapsto \varphi(\gamma) = \left( \rho_{\gamma}, \gamma \right). \]

In particular when speaking about an action of $G(F,F')$ on the graph $X_{n,d}$, we will always refer to the action defined in Proposition \ref{prop-act-Gnd-X}, restricted to $G(F,F')$. This means that $\gamma \in G(F,F')$ acts on $(f,e) \in X_{n,d}$ by \[ \gamma \cdot (f,e) = (f^{\gamma},\gamma e),\] where \begin{equation} \label{eq:action-g(f,f')} (f^{\gamma})_v = \rho_{\gamma,v} \cdot f_{\gamma^{-1}v} = \alpha (\sigma(\gamma,\gamma^{-1}v)) \cdot f_{\gamma^{-1}v}. \end{equation}

This action should not be confused with the standard action $((f_v),e) \mapsto ((f_{\gamma^{-1}v}),\gamma e)$ coming from the inclusion of $G(F,F')$ in $\aut$.



\begin{prop} \label{prop-g(f,f')-trans-som}
Let $F \leq F' \leq \Sy$, and $n$ the index of $F$ in $F'$.
\begin{enumerate}[label=(\alph*)]
	\item \label{item-g(f,f')-trans-X} The group $G(F,F')^\ast$ acts cocompactly on $X_{n,d}$. When $F$ is transitive on $\Omega$, the group $G(F,F')^\ast$ acts transitively on vertices of $X_{n,d}$.

	\item \label{item-g(f,f')-stab-X} The stabilizer of a vertex $\left( (0), e \right) \in X_{n,d}$ in $G(F,F')$ is the stabilizer of $e$ in $U(F)$. In particular the action of $G(F,F')$ on $X_{n,d}$ is proper.
	\end{enumerate}
	Therefore when $F$ is regular, the group $G(F,F')^\ast$ acts freely transitively on the vertices of $X_{n,d}$.
\end{prop}

\begin{proof}
We show that for every vertex $x = \left( (f_v), e \right)$ of $X_{n,d}$, there is $g \in G(F,F')^\ast$ such that $g \cdot x = \left( (0), e \right)$. Since $U(F)$ preserves the vertices of this form, and since the number of orbits of $U(F)^\ast \leq G(F,F')^\ast$ on $E_d$ is finite and is equal to one when $F$ is transitive \cite[\S 3.2]{BM-IHES}, statement \ref{item-g(f,f')-trans-X} will follow.

We argue by induction on the cardinality $N$ of the support of $(f_v)$. There is nothing to show if $N=0$. Assume $N \geq 1$, and let $v_0 \in V_d$ with $f_{v_0} \neq 0$ and such that $v_0$ maximizes the distance from $e$ among vertices $v$ such that $f_v \neq 0$. Let $e_0$ be the edge emanating from $v_0$ toward $e$ (if $v_0$ belongs to $e$ then $e_0=e$), and let $a \in \Omega$ be the color of $e_0$. We also denote by $T^1$ and $T^2$ the two half-trees defined by $e_0$, where $T^1$ contains $v_0$. For every $b \in \Omega$, $b \neq a$, we denote by $e_{0,b}$ the edge containing $v_0$ and having color $c(e_{0,b}) = b$, and by $T^{1,b}$ the half-tree defined by $e_{0,b}$ not containing $v_0$. 

By assumption the permutation group $F'$ preserves the $F$-orbits in $\Omega$, so we have $F' = F F_a'$. The subgroup $\alpha(F') \leq \mathfrak{S}_{n}$ being transitive, it follows from the previous decomposition that there exists $\sigma \in F_a'$ such that $\alpha(\sigma) \cdot f_{v_0} = 0$. For every $b \neq a$, we choose $\sigma_b \in F$ such that $\sigma_b(b) = \sigma(b)$, and we consider the unique element $h \in \aut$ whose local permutations are $\sigma(h,v) = 1$ if $v \in T^2$; $\sigma(h,v_0) = \sigma$; and $\sigma(h,v) = \sigma_b$ for every $v \in T^{1,b}$ and every $b \neq a$. It is an easy verification to check that $h$ is a well-defined automorphism of $T_d$, and $h \in G(F,F')$ because all but possibly one local permutations of $h$ are in $F$.

Note that $h$ fixes $e$ by construction. Write $h(x) = ((\phi_v),e)$. We claim that the support of $(\phi_v)$ has cardinality $N-1$. Since $h$ fixes $v_0$, by (\ref{eq:action-g(f,f')}) we have $\phi_{v_0} = \alpha(\sigma(h,v_0)) \cdot f_{v_0} = \alpha(\sigma) \cdot f_{v_0} = 0$. Moreover we also have $\phi_v = f_v$ for every $v$ in $T^2$ because $h$ acts trivially on $T^2$. Finally by the choice of $v_0$ we had $f_v = 0$ for every $v \neq v_0$ in $T^1$, and since $\sigma(h,v) \in F$ for all these $v$ and $\alpha(F)$ fixes $0$, we still have $\phi_v = 0$ for every $v$ in $T^1$, $v \neq v_0$. This proves the claim, and the conclusion follows by induction. 

Statement \ref{item-g(f,f')-stab-X} follows from Lemma \ref{lem-stab-vert-X}, and the last statement follows from \ref{item-g(f,f')-trans-X} and \ref{item-g(f,f')-stab-X} and the fact that $U(F)$ acts freely on $E_d$ when $F$ is semi-regular.
\end{proof}

Propositions \ref{prop-embed-g(f,f')-wr}-\ref{prop-g(f,f')-trans-som} and Lemma \ref{lem-stand-act} imply Theorem \ref{thm-coc-and-graph}. Note that if $\varphi$ is the map defined in Proposition \ref{prop-embed-g(f,f')-wr}, then the composition of $\varphi$ with the projection to $\aut$ is just the trivial inclusion of $G(F,F')$ in $\aut$, so that $G(F,F')$ is indeed irreducible in $G_{n,d}$.

Note that when $F$ is regular, we have an explicit description of a generating subset of the group $G(F,F')^\ast$ whose associated Cayley graph is $X_{n,d}$. For, fix an edge $e_0 \in E_d$, whose color is $a \in \Omega$ and whose vertices are $v_0,v_1$, and denote $x_0 = ((0),e_0) \in X_{n,d}$. For $i=0,1$, let $S_i$ be the set of $\gamma \in G(F,F')$ fixing $v_i$ and such that $\sigma(\gamma,v) \in F$ for every $v \neq v_i$, and $\sigma(\gamma,v_i)$ is non-trivial and belongs to $F \cup F_a'$. Then $S = S_0 \cup S_1$ generates $G(F,F')^\ast$, and $\mathrm{Cay}(G(F,F')^\ast,S) \rightarrow X_{n,d}$, $\gamma \mapsto \gamma x_0$, is a graph isomorphism. Moreover neighbours of type 1 (resp.\ type 2) of a vertex of $X_{n,d}$ are labeled by elements $s \in S_i$ such that $\sigma(\gamma,v_i) \in F$ (resp.\ $\sigma(\gamma,v_i) \in F_a'$). 

\section{Further comments} \label{sec-comments}

\subsection{Variations of the graphs} \label{subsec-variations}

There are possible variations in the definition of the graph $X_{n,d}$. If $\mathcal{K}_n$ is the complete graph on $n$ vertices, let $\mathcal{C}_{n,d}$ be the wreath product of the graphs $\mathcal{K}_n$ and $T_d$: the vertex set is $\Sigma_n^{(V_d)} \times V_d$, and there is an edge between $(f,v)$ and $(f',v')$ if and only if either $f=f'$ and $v,v'$ are adjacent in $T_d$, or $v=v'$ and $f(w)=f'(w)$ if and only if $w \neq v$. Equivalently, $\mathcal{C}_{n,d}$ is the Cayley graph of $C_n \wr W_d$ described in the introduction. Proposition \ref{prop-act-Gnd-X} carries over to this graph, so that the last statement of Theorem \ref{thm-intro-coc-and-graph} indeed follows from the first statement.

The reason why we considered the graph $X_{n,d}$ instead of $\mathcal{C}_{n,d}$ in \S \ref{subsubsec-irr-wr} is to obtain, under the assumption that $F$ is semi-regular, a \textit{free action} of $G(F,F')^\ast$ on the set of vertices. In the case of $\mathcal{C}_{n,d}$, the stabilizer of a vertex in $G(F,F')^\ast$ is finite, but non-trivial. We note that it might be interesting to investigate whether the generalized wreath products of graphs from \cite{Erschler-gen-wp} could provide other kind of interesting groups of automorphisms.

Yet another possibility is to take the same vertex set as $X_{n,d}$, but declaring that there is an edge between $(f,e)$ and $(f',e')$ if $e \neq e'$ share a vertex $w$ and $f_v = f_v'$ for every $v \neq w$. This graph $Z_{n,d}$ has larger degree, namely $2(d-1)n$. Again all the results proved above for $X_{n,d}$ remain true. In the case $d=2$, one may check that $Z_{n,2}$ is the Diestel-Leader graph $\mathrm{DL}(n,n)$, so that $Z_{n,d}$ may be thought of as \enquote{higher dimensional} versions of these graphs.

\subsection{Additional remarks}

We end the article by observing that the embedding from Proposition \ref{prop-embed-g(f,f')-wr} also provides examples countable ICC groups which are cocompact lattices in a strongly amenable locally compact group, and discuss this phenomenon in connection with the recent work \cite{Fr-Ta-VaFe}.

In the sequel we use the notation of the previous sections. We fix a vertex $v_0$ of the tree $T_d$, and denote by $K_d$ the stabilizer of $v_0$ in $\aut$. We also denote by $H_{n,d}$ the stabilizer of $v_0$ in $G_{n,d}$ for the projection action on $T_d$, i.e.\ $H_{n,d} = \mathfrak{S}_{n}^{V_d, \mathfrak{S}_{n-1}} \rtimes K_d$; and by $K(F,F')$ the stabilizer of $v_0$ in the group $G(F,F')$. Since $H_{n,d}$ is an open subgroup $G_{n,d}$, and since intersecting a cocompact lattice $\Gamma$ of a group $G$ with an open subgroup $O$ of $G$ provides a cocompact lattice of $O$, Propositions \ref{prop-embed-g(f,f')-wr}-\ref{prop-g(f,f')-trans-som} imply:

\begin{prop} \label{prop-K(F,F')-lattice}
Let $F \leq F' \leq \Sy$ such that $F$ is semi-regular, and let $n$ be the index of $F$ in $F'$. Then the group $K(F,F')$ embeds as a cocompact lattice in the group $H_{n,d}$.
\end{prop}

Observe that the abelian group $\oplus_{V_d} C_n$ embeds as a cocompact lattice of $\mathfrak{S}_{n}^{V_d, \mathfrak{S}_{n-1}}$, and hence as a cocompact lattice of $H_{n,d}$ since $K_d$ is a compact group. In particular the group $H_{n,d}$ is an example of a locally compact group with two cocompact lattices, one of which is abelian, and the other is an ICC group (recall that a group is ICC if every non-trivial conjugacy class is infinite). Indeed:

\begin{prop} \label{prop-K(F,F')-icc}
Let $F \lneq F' \leq \Sy$. Then $K(F,F')$ is an ICC group.
\end{prop}

\begin{proof}
This is easy. If $\gamma$ is a non-trivial element, then it moves a vertex $v$ of the tree. If $T(v)$ is the subtree consisting of vertices $w$ such that the geodesic $[v_0,w]$ contains $v$, then $\gamma T(v) \cap T(v)$ is empty since $\gamma$ also fixes $v_0$. So if $\Lambda$ is the subgroup of $K(F,F')$ that fixes pointwise the complement of $T(v)$, then $\gamma \Lambda \gamma^{-1} \cap \Lambda = 1$. But $\Lambda$ is easily seen to be infinite since $F \neq F'$, and it follows that $\gamma$ cannot centralize a finite index subgroup of $\Lambda$. So the centralizer of $\gamma$ cannot have finite index in $K(F,F')$, and $\gamma$ has an infinite conjugacy class.
\end{proof}

Recall from \cite{Gl-PF} that a group $G$ is \textbf{strongly amenable} if every proximal action of $G$ on a compact space $X$ has a fixed point. An action of $G$ on $X$ is proximal if the closure of every orbit in $X \times X$ intersects the diagonal. 

\begin{prop} \label{prop-H-sa}
For every $n,d$, the group $H_{n,d}$ is strongly amenable.
\end{prop}

\begin{proof}
Let $X$ be a compact $H_{n,d}$-space that is proximal. Since the subgroup $\mathfrak{S}_{n}^{V_d, \mathfrak{S}_{n-1}}$ is cocompact in $H_{n,d}$, its action on $X$ is also proximal \cite[II.3.1]{Gl-PF}. In particular the subgroup $\oplus \mathfrak{S}_{n}$, which is dense in $\mathfrak{S}_{n}^{V_d, \mathfrak{S}_{n-1}}$, also has a proximal action on $X$. But $\oplus \mathfrak{S}_{n}$ is an FC-group, and hence is strongly amenable \cite[II.3.2-II.4.1]{Gl-PF}. So there is a point $x$ in $X$ that is fixed by $\oplus \mathfrak{S}_{n}$, and hence also by $\mathfrak{S}_{n}^{V_d, \mathfrak{S}_{n-1}}$ by density. Since $\mathfrak{S}_{n}^{V_d, \mathfrak{S}_{n-1}}$ cannot have more than one fixed point by proximality, and since $\mathfrak{S}_{n}^{V_d, \mathfrak{S}_{n-1}}$ is a normal subgroup of $H_{n,d}$, it follows that $x$ is actually fixed by the entire $H_{n,d}$. So $H_{n,d}$ has a fixed point in $X$, and we are done.
\end{proof}

According to \cite{Fr-Ta-VaFe}, a countable ICC group is not strongly amenable. This applies to any group $K(F,F')$ where $F$ is semi-regular by Proposition \ref{prop-K(F,F')-icc}. Nonetheless by Proposition \ref{prop-K(F,F')-lattice} $K(F,F')$ embeds as a cocompact lattice in $H_{n,d}$ for $n = (F':F)$, and the latter is strongly amenable by Proposition \ref{prop-H-sa}. This shows:

\begin{prop} \label{prop-strong-am-coco}
The property of being strongly amenable does not pass from a locally compact group to a discrete cocompact subgroup.
\end{prop}

This contrasts with the case of discrete groups, for which strong amenability is inherited by subgroups, as follows from the main result of \cite{Fr-Ta-VaFe}.

As a final remark, we also observe that the above example of the group $H_{n,d}$ with an abelian cocompact lattice and an ICC cocompact lattice, also shows that two cocompact lattices in the same locally compact group may be such that one has the Choquet-Deny property while the other does not. For background and recent developments on the Choquet-Deny property, we refer to \cite{FHTP-Choquet}. 

\bibliographystyle{amsalpha}
\bibliography{simpleandwreath}

\end{document}